\documentclass[10pt, reqno]{amsart} 
\usepackage[T2A]{fontenc}
\usepackage[utf8]{inputenc}
\usepackage[english]{babel}
\usepackage[all]{xy} 
\usepackage{amsmath,amsthm, amssymb,fancyhdr,mathrsfs,color,epsfig, color, graphicx,mathtools,old-arrows,caption}
\usepackage[usenames,dvipsnames,svgnames,table]{xcolor}
\usepackage[colorlinks=true,hyperindex, pagebackref=true, citecolor=Blue,linkcolor=magenta]{hyperref}

\newtheorem{theorem}[subsection]{Theorem}
\newtheorem{lemma}[subsection]{Lemma}
\newtheorem{proposition}[subsection]{Proposition}

\newtheorem{remark}[subsection]{Remark}

\newcommand\RRR{\mathbb{R}}
\newcommand\ZZZ{\mathbb{Z}}
\newcommand{\NNN}{\mathbb{N}}
\newcommand{\id}{\mathrm{id}}

\renewcommand{\mod}{\mathrm{mod}\,}
\newcommand{\DMX}{\mathcal{D}(M,X)}
\newcommand{\SfX}{\mathcal{S}(f,X)}
\newcommand{\OfX}{\mathcal{O}(f,X)}
\newcommand{\DidMX}{\mathcal{D}_{\id}(M,X)}
\newcommand{\SidfX}{\mathcal{S}_{\id}(f,X)}
\newcommand{\OffX}{\mathcal{O}_{f}(f,X)}

\newcommand{\Off}{\mathcal{O}_f(f)}

\newcommand{\supp}{\mathrm{supp}}

\newcommand{\DM}{\mathcal{D}(M)}
\newcommand{\Sf}{\mathcal{S}(f)}

\newcommand{\inj}{\hookrightarrow}
\newcommand{\surj}{\twoheadrightarrow}
\newcommand{\xsurj}{\ar@{->>}[r]}
\newcommand{\xinj}{\ar@{^{(}->}[r]}

\begin{document}
\title{Deformations of circle-valued  functions on $2$-torus}
\author{Bohdan Feshchenko}
\address{Topology laboratory, Department of algebra and topology, Institute of Mathematics of National Academy of Science of Ukraine,
Tereshchenkivska, 3, Kyiv, 01601, Ukraine}
\email{fb@imath.kiev.ua}
\subjclass[2010]{}
\keywords{Circle-valued Morse functions, orbits, stabilizers, fundamental groups}
\date{\today}

\maketitle

\begin{abstract}
	In this paper we give an algebraic description of  fundamental groups of orbits of circle-valued smooth functions from some subspace of the space of smooth functions with isolated singularities on $2$-torus $T^2$ with respect to the action of the group of diffeomorphisms of $T^2$. 
\end{abstract}

\section{Introduction}\label{sec:Introduction}
Morse functions on manifolds are one of the main objects in mathematics nowadays. It is well-known that  analytic properties of such functions  carry information about the geometry and topology of the manifold on which they are defined \cite{Milnor:MorseTheory}.

Deformational properties of Morse functions were studied by many authors. For example,
homotopy properties of connected components of spaces of Morse functions on smooth surfaces were studied by V.~Sharko \cite{Sharko:PrIntMat:1998}, H.~Zieschang, S.~Matveev, E.~Kudryavtseva \cite{Kudryavtseva:MatSb:1999}.
Cobordism groups of Morse functions on surfaces were calculated by К.~Ikegami, О.~Saeki \cite{IkegamiSaeki:JMSJap:2003} and B.~Kalmar \cite{Kalmar:KJM:2005}.

The paper is devoted to the study of  circle-valued smooth functions from some subspace of smooth functions  with isolated singularities on smooth  compact oriented surfaces and homotopy properties of special subspaces of such functions called orbits. The main example of such functions is circle-valued Morse functions which are
natural generalizations of (ordinary) Morse  functions. 
Circle-valued Morse functions can be viewed locally as  functions but global properties of such functions are different from real-valued case.
A modern theory of circle-valued Morse  functions originates in a series of papers \cite{Novikov:1981:DAN, Novikov:1982:UMN} by S.~Novikov in 80's. It  was motivated by the study of multi-valued Lagrangians in some problems of theoretical physics and leaded him to develop a generalization of a Morse  theory  for circle-valued Morse functions and more generally a theory (now called Morse-Novikov theory) for differential $1$-forms. 
This theory has many applications, e.g., in questions of fibrations of manifolds over $S^1$ 
\cite{Ranicki:1995:T},  Lagrangian intersections \cite{FukayaOhOhtaOno:2009AMS}, knot theory \cite{VeberPazhitnovRudolf:2001:PANAA},  Seiberg-Witten  theory \cite{HutchingsLee:1999:GT, HutchingsLee:1999:T}, etc. The reader can find more on this theory and its applications in the book by A.~Pajitnov \cite{Pajitnov:2006:Gruyter}.

Recall that there is a natural action of the group of 
diffeomorphisms $\mathcal{D}(M)$ of a smooth compact surface $M$   on the space of smooth $P$-valued functions for $P = \RRR$ or $S^1$ given  by the rule:
$$
\gamma:C^{\infty}(M,P)\times \mathcal{D}(M)\to C^{\infty}(M,P),\quad \gamma(f,h) = f\circ h.
$$
We consider stabilizers $\mathcal{S}(f)$ and orbits $\mathcal{O}(f)$ of a smooth function $f\in C^{\infty}(M,P)$  with respect to the action $\gamma$   and their connected components $\mathcal{S}_{\id}(f)$ and $\mathcal{O}_f(f)$ of the identity map $\id_M$ and a connected component containing $f$ respectively in topologies induced from strong Whitney topologies on $\mathcal{D}(M)$ and $C^{\infty}(M,P)$ (see definitions in subsection \ref{subsec:Orbits-Stabilizers}).

We restrict out attention on the class of smooth  $P$-valued functions $\mathcal{F}$ on a smooth compact surface which satisfies two conditions: functions take constant values on each boundary component and near every critical point can be represented as homogeneous polynomial $\RRR^2\to \RRR$ of degree $\geq 2$ without multiple factors (class of functions $\mathcal{F}$, see subsection \ref{subsec:class-F}). It is well-known that the class $\mathcal{F}$  consists of ``generic'' functions with ``topologically generic'' singularities(\cite[Section 3,4]{Maksymenko:2021:review}, subsection \ref{subsec:class-F}). Thus our restriction is insignificant, since the class $\mathcal{F}$ is ``wide'' enough.

S. Maksymenko \cite{Maksymenko:AGAG:2006, Maksymenko:MFAT:2010, Maksymenko:ProcIM:ENG:2010, Maksymenko:UMZ:ENG:2012, Maksymenko:DefFuncI:2014} showed that if $f:M\to P$ is a smooth function from $\mathcal{F}$ and $f$ has at least one saddle point, then $\pi_n\mathcal{O}_f(f) = \pi_nM$ for $n\geq 3$, $\pi_2\mathcal{O}_f(f) = 0$ and for $\pi_1\mathcal{O}_f(f)$ there is a short exact\footnote{Throughout the text $\inj$ and $\surj$ mean mono- and epimorphism respectively.} sequence of groups
\begin{equation}
	\xymatrix{
		\pi_1\mathcal{D}_{\id}(M) \ar@{^{(}->}[r]^{\zeta_1} & \pi_1\mathcal{O}_f(f) \ar@{->>}[r]^{\partial_1} & \pi_0\mathcal{S}'(f), 
	}
	\tag{2.3}
\end{equation}
where $\mathcal{S}'(f)$ is the group of  $f$-preserving diffeomorphisms of $M$ which are isotopic to the identity map, see Eq. \eqref{eq:S'}.
This exact sequence is a non-trivial part of long exact sequence of homotopy groups of some fibration $\zeta_f:\mathcal{D}_{\id}(M)\to \mathcal{O}_f(f)$ (Theorem \ref{thm:homotopy-gneral}).
As the consequence if $M$ is aspherical surface then all homotopy information is contained in $\pi_1\mathcal{O}_f(f)$, and the orbit  $\mathcal{O}_f(f)$ itself is an Eilenberg–MacLane space $K(\pi_1\mathcal{O}_f(f),1).$ 
So there are two natural questions about homotopy properties of $\mathcal{O}_f(f)$: description of an algebraic structure of groups $\pi_1\mathcal{O}_f(f)$ and homotopy type of $\mathcal{O}_f(f).$

An algebraic structure of $\pi_1\mathcal{O}_f(f)$ ``partially'' depends on homotopy properties of $\mathcal{D}_{\id}(M)$ which were studied in  \cite{EarleEells:BAMS:1967,EarleEells:DG:1970, EarleSchatz:DG:1970, Gramain:ASENS:1973}.
It is well-known that 
if $M$ is a closed compact and oriented surface of genus $\geq 2$, then $\mathcal{D}_{\id}(M)$ is contractible \cite{EarleEells:BAMS:1967,EarleEells:DG:1970, EarleSchatz:DG:1970, Gramain:ASENS:1973}, and so an epimorphism $\partial_1$ from \eqref{seq:homotopy-pi1} is an isomorphism. 
Thus in this case the question on an algebraic structure of $\pi_1\mathcal{O}_f(f)$ reduces to the study of an algebraic structure of groups $\pi_0\mathcal{S}'(f)$, which are easier to compute.

It is also known that the group $\mathcal{D}_{\id}(T^2)$ for $2$-torus is not contractible, so the image of $\pi_1\mathcal{D}_{\id}(T^2)\cong \ZZZ^2$ is non-trivial in $\pi_1\mathcal{O}_f(f)$.  The sequence \eqref{seq:homotopy-pi1} for functions on $T^2$ in general does not split \cite{MaksymenkoFeshchenko:2014:HomPropTree}. So for this case we need an additional study.

An algebraic structure of $\pi_1\mathcal{O}_f(f)$ for real-valued functions from the class $\mathcal{F}$ on $2$-torus $T^2$ was studied in the series of papers by S.~Maksymenko and the  author \cite{MaksymenkoFeshchenko:2014:HomPropCycle, MaksymenkoFeshchenko:2014:HomPropTree, MaksymenkoFeshchenko:2015:HomPropCycleNonTri, Feshchenko:2014:HomPropTreeNonTri}.
The sufficient conditions when sequence \eqref{seq:homotopy-pi1}  splits were proved in
\cite{MaksymenkoFeshchenko:2014:HomPropTree}. 
It was shown that the question on an algebraic structure of $\pi_1\mathcal{O}_f(f)$ is reduced to studying algebraic structure of orbits and stabilizers for  restrictions of a given function to subsurfaces of $T^2$ being 2-disks and cylinders, for which this structure is known \cite{Maksymenko:DefFuncI:2014}. These cylinders and $2$-disks are in some sense ``building blocks'' which carry  ``combinatorial symmetries'' of the function.

The group $\pi_0\mathcal{S}'(f)$ contains a lot of information about an algebraic structure of $\pi_1\mathcal{O}_f(f)$ and on homotopy type of $\mathcal{O}_f(f)$. Note that each diffeomorphism from $\mathcal{S}'(f)$ induces an automorphism of Kronrod-Reeb  graph $\Gamma_f$ of $f$. The group $\pi_0\mathcal{S}'(f)$ contains a free abelian subgroup generated by Dehn twists which induce trivial action on $\Gamma_f$ and the corresponding quotient group $G(f)$ contains ``discrete combinatorial'' symmetries of the function $f$. So the group $\pi_0\mathcal{S}'(f)$ can be viewed as ``homeotopy'' group for $f$-preserving and isotopic to the identity diffeomorphisms with $G(f)$ as its non-trivial counterpart. 

It is known that the group $G(f)$ ``controls'' the homotopy type of $\mathcal{O}_f(f)$.  In particular, if $f$ is generic, then the group $G(f)$ is trivial, and $\mathcal{O}_f(f)$ is homotopy equivalent to an $m$-torus $T^m$ if $M\neq S^2$, $M\neq \RRR P^2$,  to $S^2$ if $M = S^2$ and $f$ has only two critical points, and to $\mathrm{SO}(3)\times T^m$ otherwise, for some $m\geq 0$ depending on $f$. E.~Kudryavtseva \cite{Kudryavtseva:MathNotes:2012, Kudryavtseva:MatSb:2013} calculated the homotopy types of connected components of the space of Morse functions on compact surfaces and generalized the result on homotopy types of orbits $\mathcal{O}_f(f)$ when the group $G(f)$ is non-trivial.

It is also known that if $f:M\to P$ has exactly $n$ critical points, then $\Off$ is homotopy equivalent to some covering space of $n^{\mathrm{th}}$ configuration space of $M$ and $\pi_1\mathcal{O}_f(f)$ is a subgroup of $n^{\mathrm{th}}$ braid group $B_n(M)$ \cite[Theorem 2, Corollary 4]{Maksymenko:Travaux:2007}. 
A good overview  the reader can find in \cite{Maksymenko:2021:review,Maksymenko:DefFuncI:2014} where these results are presented in the form of so-called crystallographic and Bieberbach sequences.

{\bf Our main goal} is to generalize  our results on algebraic structure of $\pi_1\mathcal{O}_f(f)$ to the case of circle-valued  functions from the class $\mathcal{F}$ on $T^2.$ 

Notice that if $f:T^2\to S^1$ is a smooth function without critical points, then $f:T^2\to S^1$ is a locally trivial fibration and the homotopy types of orbits and stabilizers are known, see \cite[Theorem 1.9]{Maksymenko:AGAG:2006}, so we will always assume that all functions have at least one critical point.

\subsection*{General overview of results.} First we show (Theorem \ref{prop:circle-null-orbit}) if $f:M\to S^1$ is a null-homotopic function from $\mathcal{F}$ then $\mathcal{O}_f(f)$ is homeomorphic to $\mathcal{O}_{\tilde{f}}(\tilde{f})$, where $\tilde{f}:M\to \RRR$ is also a smooth function from $\mathcal{F}$ which is a lift of $f$ with respect to the universal cover $p:\RRR\to S^1.$ So for null-homotopic functions the problem in hand is completely reduced to the real-valued case, which is known.

Then we show that for circle-valued functions with isolated singularities on $2$-torus Kronrod-Reeb graphs are trees or contain a unique cycle, and if $f$ is not null-homotopic, then these graphs are not trees (Lemma \ref{lemma:Gamma_f-S^1}).

In our  main result (Theorem \ref{thm:main-th}) we use an algebraic construction  -- wreath product (subsection \ref{subsec:wreath}) -- to  describe  an algebraic structure of $\pi_1\mathcal{O}_f(f)$ for functions from $\mathcal{F}$ on $T^2$ whose graphs contain a cycle via groups  of connected components of stabilizers of the restrictions of a function $f$ onto cylinders.

It is well-known that $\pi_1\mathcal{D}_{\id}(T^2)$ is free abelian group of the rank $2$ (\cite{EarleEells:DG:1970, Gramain:ASENS:1973}); we define these two generators $\mathbf{L}$ and $\mathbf{M}$ adapted to some coordinate system on $T^2$ (see Eq. \eqref{eq:IsotopiesML}). The image of $\mathbf{L}$ in $\pi_1\mathcal{O}_f(f)$  plays an important role in the proof of Theorem \ref{thm:main-th}, but the image of $\mathbf{M}$ is ``invisible'' in our description of $\pi_1\mathcal{O}_f(f)$. Proposition \ref{theorem:kernel-phi-M} contains the information about  image of $\mathbf{M}$ in $\pi_1\mathcal{O}_f(f).$ 

The obtained results  will be used to further study the  orbits of smooth functions as well as their relationships with braid groups  on surfaces mentioned above or more generally Artin groups.

\subsection*{Structure of the paper.} In Section \ref{sec:Definitions} we give  definitions of stabilizers and orbits of smooth functions with values in $P$ on surfaces (subsection \ref{subsec:Orbits-Stabilizers}) and discuss their homotopy properties (Theorem \ref{thm:homotopy-gneral}) for functions from the class $\mathcal{F}$.

Section \ref{sec:Auxiliary} contains auxiliary constructions such as null-homotopic functions to $S^1$ and their orbits (subsection \ref{subsec:Null-homotopic}), topological properties of graphs of circle-valued functions with isolated singularities on $T^2$ (subsection \ref{subsec:Graphs}) and wreath product of special type (subsection \ref{subsec:wreath})  needed to state our main result. 

Then we state our main result -- Theorem \ref{thm:main-th} in Section \ref{sec:main-result}.
Section \ref{sec:Additional-construction-definitions} contains some additional constructions and result needed to the proof of Theorem \ref{thm:main-th} in Section \ref{sec:Proof-main}.

Finally, we study the ``place'' of the generator $\mathbf{M}$ of $\pi_1\mathcal{D}_{\id}(T^2)$ in $\pi_1\mathcal{O}_f(f)$ (Section \ref{sec:isitopyM}).

\section{Definitions and useful facts}\label{sec:Definitions}
\subsection{Orbits and stabilizers of smooth functions}\label{subsec:Orbits-Stabilizers}
Let $M$ be a smooth compact surface, $X$ be a closed (possible empty) subset of $M$. By $P$ we also denote $\RRR$ or $S^1$. The group $\DMX$ of diffeomorphisms of $M$ fixed on $X$ acts from the right on the space of smooth maps $C^{\infty}(M,P)$ by the  rule
$$
\gamma: C^{\infty}(M,P)\times \DMX\to C^{\infty}(M,P),\qquad \gamma(f,h) = f\circ h.
$$
With respect to $\gamma$ we denote by 
\begin{align*}
	\SfX &= \{ h\in \DMX\,|\, f\circ h = f\},\\
	\OfX &= \{f\circ h\,| h\in\DMX\}
\end{align*}
the {\it stabilizer} and the {\it orbit} of $f\in C^{\infty}(M,P).$ Endow strong Whitney $C^{\infty}$-topologies on $C^{\infty}(M,P)$ and $\DMX;$ then for a map $f\in C^{\infty}(M,P)$ these topologies induce some topologies on $\SfX$ and $\OfX$. We denote by $\DidMX$, $\SidfX$ connected components of the identity map $\DMX$ and $\SfX$ respectively, and by $\OffX$ a connected component of $\OfX$ containing $f$. If $X = \varnothing$ we omit the symbol ``$\varnothing$'' from our notation, i.e., we will write $\DM$ and $\Sf$ instead of $\mathcal{D}(M,\varnothing)$ and $\mathcal{S}(f,\varnothing)$ and so on. We also set
\begin{equation}\label{eq:S'}
	\mathcal{S}'(f,X) = \mathcal{S}(f,X)\cap\DidMX.
\end{equation}
The group $\mathcal{S}'(f,X)$ consists of isotopic to the identity $f$-preserving diffeomorphisms fixed on $X$.

\subsection{The space $\mathcal{F}(M,P)$}\label{subsec:class-F}
Let $\mathcal{F}(M,P)$ be a  subset of $C^{\infty}(M,P)$ satisfying the following conditions:
\begin{enumerate}
	\item the map $f$ takes constant values at each boundary component of $M$,
	\item for every critical point $z$ of $f$ there are local coordinates in which $f$ is a homogeneous polynomial $\RRR^2\to \RRR$ of degree $\geq 2$ without multiple factors.  
\end{enumerate}
Notice that the class $\mathcal{F}$ is ``natural'' and ``generic'' class of smooth functions.
It is well-know that every $f\in\mathcal{F}(M,P)$ has only isolated critical points, and thus the set of critical points of $f$ is finite. For instance the space of $P$-valued Morse functions, i.e., smooth maps $f:M\to P$ which has only non-degenerate critical points, satisfying condition (1) is a subspace of $\mathcal{F}(M,P).$ 
From the other hand, if a smooth function $f:M\to P$ has only isolated critical points, then by theorem proved by P.~T.~Church and J.~G.~Timourian \cite{ChurchTimourian} and independently by O.~Prishlyak \cite{Prishlyak:2002}, the local topological structure of level sets near any critical point can be realized by level sets of homogeneous polynomial without multiple factors.
So the space $\mathcal{F}(M,P)$ consists of ``generic'' maps with ``topologically generic'' critical points, see \cite[Section 3,4]{Maksymenko:2021:review}. 

\subsection{$f$-adapted manifolds}\label{subsec:f-adapted} 
To state the results about homotopy properties of orbits and stabilizers for functions from the class $M$ on surfaces we need the notion called $f$-adapted manifolds.

Let $f:M\to P$ be a smooth function from $\mathcal{F}$ on compact surface $M.$
A connected component of a level set $f^{-1}(c)$, $c\in P$ is also called a {\it leaf} of $f$. A leaf is called {\it regular} if it contains no critical points and {\it critical} otherwise.

Let $K$ be a (regular or critical) leaf of $f$. For $\varepsilon>0$ let $N_{\varepsilon}$ be a connected component of $f^{-1}[c-\varepsilon, c+\varepsilon]$ containing $K.$ Then $N_{\varepsilon}$ is called an {\it$f$-regular neighborhood} of $K$ if $\varepsilon$ is so small that $N_{\varepsilon}-K$ contains no critical points and no boundary components.

A submanifold $X\subset M$ is called {\it$f$-adapted} if $X = \cup_{i = 1}^a A_i$, where each $A_i$ is either a critical point of $f$, or  a regular leaf of $f$, or an $f$-regular neighborhood of some (regular or critical) leaf of $f$. Note that if $X$ is a $f$-adapted  subsurface, then $f|_X:X\to P$  belongs to $\mathcal{F}(X,P),$ see \cite{Maksymenko:2021:review}.

For a set $X$ we denote by $|X|$ the number of point in $X$; if $X$ is an infinite set we put $|X| = \infty$.

\subsection{Homotopy properties of orbits and stabilizers}

The following theorem describes the general homotopy properties of  orbits.
\begin{theorem}[\cite{Sergeraert:ASENS:1972,Maksymenko:AGAG:2006,Maksymenko:UMZ:ENG:2012,Maksymenko:OsakaJM:2011}]\label{thm:homotopy-gneral}
	Let $f\in\mathcal{F}(M,P)$ be a  function on a smooth compact surface $M$ and $X$ be an $f$-adapted submanifold. Then the following statements hold.
	\begin{enumerate}
		\item The map 
		\begin{equation}\label{eq:Serre}
			\zeta_f:\DMX\to \OfX,\qquad \zeta_f(h) = f\circ h
		\end{equation}
		is a locally trivial principal fibration with the fiber $\SfX$. The restriction $\zeta_f|_{\DidMX}:\DidMX\to \OffX$ is also a locally trivial principal fibration with the fiber $\mathcal{S}'(f,X).$
		The orbit $\mathcal{O}_f(f,X)$ is a Fr\`echet manifold, so it has a homotopy
		type of a CW complex.
		\item $\OffX =\mathcal{O}_f(f, X\cup \partial M)$,  and so $$\pi_k\OffX \cong \pi_k\mathcal{O}_f(f, X\cup \partial M)$$ for $k\geq 1.$
		\item Suppose that either $f$ has a saddle point or $M$ is non-orientable surface. Then $\mathcal{S}_{\id}(f)$ is contractible, $\pi_k\mathcal{O}_f(f) \cong \pi_kM,$ $k\geq 3$, $\pi_2\mathcal{O}_f(f) = 0$, and for $\pi_1\mathcal{O}_f(f)$ the following short sequence of groups
		\begin{equation}\label{seq:homotopy-pi1}
			\xymatrix{
				\pi_1\mathcal{D}_{\id}(M) \ar@{^{(}->}[r]^{\zeta_1} & \pi_1\mathcal{O}_f(f) \ar@{->>}[r]^{\partial_1} & \pi_0\mathcal{S}'(f)
			}
		\end{equation}
		is exact, where $\zeta_1$ is a homomorphism induced by $\zeta_f$ and $\partial_1$ is a boundary map for long exact sequence of the fibration $\zeta_f$.  Moreover, $p(\pi_1\mathcal{D}_{\id}(M))$ contains in the center of $\pi_1\mathcal{O}_f(f).$
		\item If 
		$
		\chi(M) < |X|
		$
		then $\pi_1\mathcal{D}_{\id}(M,X)$ is contractible (in particular, when $X$ is finite or when $\chi(M)<0$), $\pi_k\mathcal{O}_f(f,X) = 0$ for $k\geq 2$ and the boundary map 
		$$
		\xymatrix{
			\pi_1\mathcal{O}_f(f,X) \ar[r]^{\partial_1}_{\cong} &  \pi_0\mathcal{S}'(f,X)
		}
		$$
		is an isomorphism. 
		
	\end{enumerate}
\end{theorem}
Notice that a sequence \eqref{seq:homotopy-pi1} is nonzero part of a long exact sequence of homotopy groups of the fibration $\zeta_f$. We briefly recall the definition of $\partial_1$. Let $\tilde{f}$ be a loop in $\mathcal{O}_f(f)$ based in $f:M\to P$, i.e., $\tilde{f}:[0,1]\to \mathcal{O}_f(f)$ with $\tilde{f}_0 = \tilde{f}_1 = f.$ Then by (1) of Theorem \ref{thm:homotopy-gneral} there exists a path $h:[0,1]\to \mathcal{D}_{\id}(M)$ such that 
$$
\tilde{f}_t = f\circ h_t,\qquad h_0 = \id_{M},\qquad h_1 = h\in\mathcal{S}'(f).
$$
We denote by $[\tilde{f}]$ and $[h]$ the corresponding homotopy classes $\tilde{f}$ and $h$ in $\pi_1\mathcal{O}_f(f)$ and $\pi_0\mathcal{S}'(f)$ respectively. Then the boundary homomorphism is defined by $\partial_1[\tilde{f}] = [h].$

Moreover there is an isomorphism
\begin{equation}\label{eq:varkappa}
	\varkappa:\pi_1(\mathcal{D}_{\id}(M),\mathcal{S}'(f))\longrightarrow \pi_1\mathcal{O}_f(f),\qquad \varkappa:[\{h_t\}]\longmapsto [\{f\circ h_t\}],
\end{equation}
where $h_t:M\to M$, $t\in [0,1]$ is an isotopy of $M$ with $h_0 = \id_{M}$ and $h_1\in\mathcal{S}'(f)$.

\section{Auxiliary constructions}\label{sec:Auxiliary}
\subsection{Null-homotopic functions from $\mathcal{F}(M,S^1)$ and their orbits}\label{subsec:Null-homotopic}
A function $f:M\to S^1$ homotopic to a constant map will be called {\it null-homotopic}. 

It is easy to see that homotopy properties of null-homotopic  functions from $\mathcal{F}(M,S^1)$ are the same as for its ``universal'' lift, an ordinary  function from $\mathcal{F}(M,\RRR)$ arising from a universal cover of $S^1$. To be more precise, consider a universal covering map $p:\RRR\to S^1$ given by $p(t) = e^{2\pi i t}$. Then by a lifting property for maps \cite[Proposition 1.33]{Hatcher:2002:Cambridge}, there exist a unique smooth function $\tilde{f}:M\to \RRR$ (up to a choice of appropriate triple of points) such that $p\circ \tilde{f} = f$. Note that the function $\tilde{f}$ is also a  function from $\mathcal{F}(M,\RRR)$. This function $\tilde{f}$  will be called a {\it universal lift} of $f$.
The fact that homotopy properties of $\mathcal{O}_{{f}}({f})$ are the same as for $\mathcal{O}_{\tilde{f}}(\tilde{f})$ follows from the proposition.
\begin{proposition}\label{prop:circle-null-orbit}
	Let $f\in\mathcal{F}(M,S^1)$ be a null-homotopic  function and $\tilde{f}:M\to \RRR$ be its universal lift. Then $\mathcal{O}(f)$ and 
	$\mathcal{O}(\tilde{f})$ are homeomorphic, and hence $\mathcal{O}_f(f) \cong \mathcal{O}_{\tilde{f}}(\tilde{f}).$
\end{proposition}
\begin{proof}
	Recall that $\zeta_{f}:\mathcal{D}(M)\to \mathcal{O}(f)$ and $\zeta_{\tilde{f}}:\mathcal{D}(M)\to\mathcal{O}(\tilde{f})$ are locally trivial principal fibrations with fibers $\mathcal{S}(f)$ and $\mathcal{S}(\tilde{f})$ respectively. Then $\zeta_f$ decomposes as
	the composition 
	$$
	\xymatrix@R=2em{
		\mathcal{D}(M) \ar[rd]_{\alpha_f} \ar[r]^{\zeta_f} & \mathcal{O}(f)\\
		&   \mathcal{D}(M)/{\mathcal{S}(f)} \ar[u]_{\beta_f}^{\cong},
	}
	$$
	where $\alpha_f(h) = \overline{h}$ is an open quotient map, and $\beta_f(\overline{h}) = \zeta_f(h)$ is an induced by $\zeta_f$ homeomorphism, where $\overline{h}$ is a coset of $h\in\mathcal{D}(M)$ modulo $\mathcal{S}(f)$. The same decomposition holds for $\zeta_{\tilde{f}}$. So $\mathcal{O}(f)$ and $\mathcal{O}(\tilde{f})$ are homeomorphic to quotient groups $\mathcal{D}(M)/\mathcal{S}(f)$ and $\mathcal{D}(M)/\mathcal{S}(\tilde{f})$ respectively. S.~Maksymenko \cite[Lemma 5.3]{Maksymenko:DefFuncI:2014} showed that if $f$ is null-homotopic then there is a homeomorphism $\mathcal{S}(f) \cong \mathcal{S}(\tilde{f})$. Therefore quotient spaces $\mathcal{D}(M)/\mathcal{S}(f)$ and $\mathcal{D}(M)/\mathcal{S}(\tilde{f})$ are homeomorphic, which implies that $\mathcal{O}(f)$ is homeomorphic to $\mathcal{O}(\tilde{f}),$ and so $\mathcal{O}_f(f)\cong \mathcal{O}_{\tilde{f}}(\tilde{f}).$
\end{proof}
By Proposition \ref{prop:circle-null-orbit} homotopy properties of $\mathcal{O}_f(f)$ for a null-homotopic  function $f\in\mathcal{F}(M,S^1)$ are the ``same'' as for $\mathcal{O}_{\tilde{f}}(\tilde{f})$ where $\tilde{f}$ is a universal lift of $f$. Therefore this case is completely reduced  to the real-valued case, and so we will focus  on the case when $f$ is not null-homotopic.

\subsection{Graphs of $P$-valued Morse functions}\label{subsec:Graphs}
Kronrod-Reed graphs are  important tools to study smooth functions since they carry a lot of informations about its combinatorial structure. We recall this definition.
Let $f\in\mathcal{F}(M,P)$ be a  function on a smooth compact oriented surface $M$ and $c\in P$. 
Let $\Xi$ be a partition of $M$ into connected components of level sets (leaves, see subsection \ref{subsec:f-adapted}) of $f$. It is well-known that the quotient space $M/\Xi$ denoted by $\Gamma_f$ has a structure of an $1$-dimensional CW complex called a {\it Kronrod-Reeb graph} of $f$, or simply a {\it  graph} of $f$ \cite{Reeb:ASI:1952}.

Let $q_f: M\to \Gamma_f$ be a quotient map. Then the map $f:M\to P$ can be presented as the composition of the projection $q_f$ onto $\Gamma_f$ and the map $f_{\Gamma}$ induced by $f$:
\begin{equation}\label{eq:graph-sequence}
	\xymatrix{
		f = f_{\Gamma}\circ q_f: & M\ar[r]^{q_f} & \Gamma_f \ar[r]^{f_{\Gamma}} & P.
	}
\end{equation}

\subsection*{Graphs of $S^1$-valued function on $T^2$}
Graphs of ordinary Morse functions on $T^2$ were studied in \cite{KravchenkoFeshchenko:2020:MFAT}. 
For functions  with isolated singularities on $2$-torus the following lemma holds true.
\begin{lemma}[cf. Lemma 3.1 \cite{Feshchenko:DefFuncI:2019}]\label{lemma:Gamma_f-S^1} Let $f:T^2\to S^1$ be a function  with isolated singularities.
	
	{\rm(1)} The map $q_f^*:\pi_1T^2\to \pi_1\Gamma_f$ induced by $q_f$ is an epimorphism with a non-zero kernel. So $b_1(\Gamma_f)\leq 1$, i.e., $\Gamma_f$ can be either a tree or contains a unique cycle.
	
	{\rm(2)} If $f$ is not null-homotopic, then $\Gamma_f$ is not a tree.
\end{lemma}
It is well-known that the inequality $b_1(\Gamma_f)\leq 1$ holds for  generic circle-valued Morse functions on $2$-torus \cite[Theorem 3.7]{RezendeLedesmaatallL2018} and for arbitrary circle-valued Morse functions on $2$-torus this inequality follows from  Morse form foliation theory, see \cite[Theorem 2.1, Theorem 3.1, Example 3.2]{Gelbukh2009OnTS}.
\begin{proof}[Proof of Lemma \ref{lemma:Gamma_f-S^1}]
	(1) This statement for ordinary Morse functions is proved in \cite[Lemma 3.1]{Feshchenko:DefFuncI:2019} and holds true for functions with isolated singularities. The reader can verify it step-by-step as in \cite[Lemma 3.1]{Feshchenko:DefFuncI:2019}.
	
	(2) Assume the converse holds and consider the sequence of fundamental groups induced by \eqref{eq:graph-sequence}:
	$$
	\xymatrix@R=1.5em{
		\pi_1 T^2 \ar@{=}[d] \ar[r]^{q_f^*} & \pi_1 \Gamma_f \ar@{=}[d]	\ar[r]^{f_{\Gamma}^*} & \ar@{=}[d]  \pi_1 S^1\\
		\ZZZ^2\ar[r]^{q_f^*} & 0\ar[r]^{f^*_{\Gamma}} & \ZZZ
	}
	$$
	The homomorphism $f_{\Gamma}^*$ is a zero-map, so the homomorphism $f^* = f_{\Gamma}^*\circ  q_{f}^*$ induced by $f$ is also a zero map. This leads to the contradiction that $f$ is not null-homotopic, so $\Gamma_f$ contains a cycle.
\end{proof}

The case of  functions $f:T^2\to P$ from the class $\mathcal{F}$ whose graphs are trees is a special case -- it is only possible when $f$ is null-homotopic.  So in further text we will study only the case of  functions whose graphs contain a  cycle.  Examples of such functions are given in Figure \ref{fig:functions-s^1}.
\begin{figure}[h]
	\centering
	\includegraphics[width=8cm]{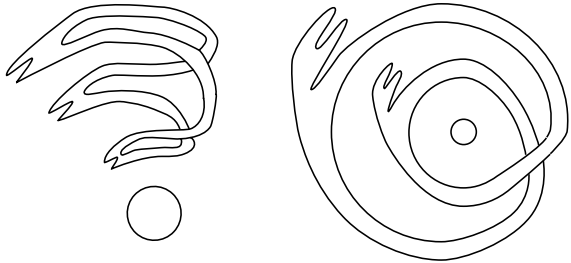}
	\caption{Null-homotopic  (left) and {\it not} null-homotopic (right) functions $T^2\to S^1$}
	\label{fig:functions-s^1}
\end{figure}
Note that  functions in  Figure \ref{fig:functions-s^1}  allow some  ``combinatorial symmetries'' preserving the given functions which play an essential role in the description of $\pi_1\mathcal{O}_f(f)$.

\subsection{Wreath products}\label{subsec:wreath}
To state our main result we need a notion of wreath product of groups of a special kind.
Let $G$ be a group and $n\geq 1$ be an integer. A semi-direct product $G^n\rtimes\ZZZ$ with respect to a non-effective $\ZZZ$-action $\alpha$  on $G^n$ by cyclic shifts
$$
\alpha(b_0,b_1,\ldots, b_{n-1}; k) = (b_k, b_{1+k}, \ldots, b_{n+k-1}),
$$
where all indexes are taken modulo $n$,
will be denoted by $G\wr_n\ZZZ$ and called a {\it wreath product} of $G$ with $\ZZZ$ under $n.$
Notice that this definition differs from the standard one, \cite{Meldrum:Longman:1995}.

\section{Main result}\label{sec:main-result}
The following theorem describes an algebraic structure of $\pi_1\mathcal{O}_f(f)$ for  functions $f:T^2\to P$ from the class $\mathcal{F}$ whose graphs contain a cycle. It will be proved in Section \ref{sec:Proof-main}.
\begin{theorem}\label{thm:main-th}
	Let $f$ be a function from $\mathcal{F}(T^2,P)$ with at least one critical point and whose graph $\Gamma_f$ contains a cycle. Then there exist a cylinder $Q\subset T^2$ and $n\in \NNN$ such that $f|_Q\in\mathcal{F}(Q,P)$  and there is an isomorphism
	$$
	\pi_1\mathcal{O}_f(f) \cong \pi_0\mathcal{S}'(f|_Q,N(\partial Q))\wr_n\ZZZ,
	$$ 
	where $N(\partial Q)$ is some $f$-regular neighborhood of $\partial Q$.
\end{theorem}
\begin{remark}
	{\rm
		Theorem \ref{thm:main-th} generalizes the main result of the paper \cite{MaksymenkoFeshchenko:2015:HomPropCycleNonTri}. To prove \cite[Theorem 1.6]{MaksymenkoFeshchenko:2015:HomPropCycleNonTri} we mainly use ``local'' technique, i.e.,  properties of diffeomorphisms of subsurfaces of $T^2$ and  ``glue'' them together to obtain global ones. So Theorem \ref{thm:main-th}  can be proved step-by-step by the same arguments and strategy. In the present paper we give more straightforward proof of this result separating algebraic methods from topological ones. Proofs of some known facts will be given only for the sake of completeness.
	}
\end{remark}

\section{Additional constructions and definitions}\label{sec:Additional-construction-definitions}
\subsection{Curves on $T^2$}
Let $f$ be a function from $\mathcal{F}(T^2, P)$ whose graph $\Gamma_f$ contains a unique cycle denoted by $\Lambda$, let also $q_f:T^2\to \Gamma_f$ be a projection induced by $f.$ Let $z$ be a point in $\Lambda$, $c = f(q^{-1}_f(z))$ be a point in $S^1$, and $C$ be a regular connected component of $f^{-1}(c)$. Note that $f^{-1}(c)$ consists of finitely many connected components and it is invariant under the action of $\mathcal{S}'(f)$.  We set $\mathcal{C} = \{ h(C)\,|\, h\in \mathcal{S}'(f)\}$.   Since $\mathcal{C}$ has finite cardinality we can cyclically enumerate elements of the set 
$\mathcal{C} = \{ C_0 = C, C_1, C_2, \ldots, C_{n-1} \}$ for some  $n\in \NNN.$ Curves from  $\mathcal{C}$ are mutually disjoint and do not separate $T^2$, and each pair $C_i$ and $C_{i+1}$ bounds a cylinder $Q_i\subset T^2$, where all indexes are taken modulo $n$.

\subsection{$f$-regular neighborhoods of curves}\label{ss:f-adapter-neighborhood}
We regard $S^1$ and $T^2$ as a quotient spaces $\RRR/\ZZZ$ and $\RRR^2/\ZZZ^2$ respectively.
By a proper choose of coordinates on $T^2$ one can assume that the following conditions hold:
\begin{itemize}
	\item $C_i = \big\{\frac{i}{n}\big\}\times S^1\subset \RRR^2/\ZZZ^2 =  T^2,$ so we can regard each curve $C_i$ as a meridian of $T^2$, and the curve $C'=\{0\}\times S^1$ as a longitude of $T^2$.
	\item there exists $\varepsilon>0$ such that  for all $t\in (\frac{i}{n} - \varepsilon, \frac{i}{n} + \varepsilon)$ the curve $\big\{t\big\}\times S^1$ is regular connected component of some level set of $f$.
\end{itemize}
This assumption makes possible to define $f$-regular neighborhoods  of curves from $\mathcal{C},$ see Subsection \ref{subsec:f-adapted}.
So an $f$-regular neighborhood $V$ of a curve $C$ is saturated neighborhood which has a cylindrical structure.

In this place we fix two families of $f$-regular neighborhoods $V_i$ and $W_i$ of $C_i$, $i = 0,\ldots, n-1$ needed in the further text, so that  $V_i \cap V_j = \varnothing$, $V_i\subset \mathrm{Int}(W_i)$ for $i\not=j$ and for each $i,j$ there exists $h\in\mathcal{S}'(f)$ such that $h(V_i)=V_j$.
In particular,  unions
$\mathsf{V} = \bigcup_{i = 0}^{n-1}V_i$ and $\mathsf{W} = \bigcup_{i = 0}^{n-1}W_i$
are $\mathcal{S}'(f)$-invariant.

\subsection{Generators of $\pi_1\mathcal{D}_{\id}(T^2)$}
Let $\mathbf{L}, \mathbf{M}:T^2\times [0,1]\to T^2$ be two isotopies defined by 
\begin{equation}\label{eq:IsotopiesML}
	\mathbf{L}(x,y,t) = (x + t\,\mod 1, y),\qquad \mathbf{M}(x,y,t) = (x, y+ t\,\mod 1)
\end{equation}
for $x\in C'$, $y\in C_k$, $k = 0,1,\ldots, n-1$. Geometrically $\mathbf{L}$ is a rotation of $T^2$ along its longitude, and $\mathbf{M}$ is a rotation along meridians.  Isotopies $\mathbf{L}$ and $\mathbf{M}$ can be regarded as loops in $\mathcal{D}_{\id}(T^2)$ It is well known that  $\mathbf{L}$ and $\mathbf{M}$ commute and
$\pi_1\mathcal{D}_{\id}(T^2) = \langle \mathbf{L}\rangle \times \langle\mathbf{M}\rangle$, see \cite{EarleEells:DG:1970,Gramain:ASENS:1973}.

\subsection{Dehn twists and slides along curves from $\mathcal{C}$} Let $Q = S^1\times [0,1]$ be a cylinder and $C$ be the curve $S^1\times \{0\}$, and $\alpha,\beta:[-1,1]\to [0,1]$ be two smooth functions such that 
$$
\alpha(x) = \begin{cases}
0,& x\in [-1,-1/2],\\
1,& x\in [1/2,1],
\end{cases}
\qquad
\beta(x) = \begin{cases}
0, & x\in [-1,-2/3]\cup [2/3,1],\\
1, & x\in [-1/3,1/3].
\end{cases} 
$$
Define two diffeomorphisms of $Q$ by formulas:
$$
\tau(z,t) = (z e^{\alpha(t)},t)\qquad \theta(z,t) = (ze^{\beta(t)}, t),\qquad (z,t)\in Q;
$$
the diffeomorphisms $\tau$ and $\theta$ are called a {\it Dehn twist} and a {\it slide} along $C = S^1\times \{0\}.$ Note that $\tau$ is fixed on some neighborhood of $\partial Q$, and $\theta$ is fixed on some neighborhood of $C\cup \partial Q$. A diffeomorphism of a smooth surface $M$ supported in some cylindrical neighborhood of a simple closed and  two-sided curve $C\subset M$ isotopic to a Dehn twist with respect to the boundary of this neighborhood will be called a Dehn twist along $C$ on $M$. Similarly the notion of slide along $C$ can be extended to the case of surfaces.

Recall that a vector field $F$ on a smooth oriented surface $M$ is called {\it Hamiltonian-like} for a  function $f$ from the class $\mathcal{F}$ if the following conditions satisfied:
\begin{itemize}
	\item singular points of $F$ correspond to critical points of $f$,
	\item $f$ is constant along $F$,
	\item Let $z$ be a critical point of $f$. Then there exists a local coordinate system $(x,y)$ such that $f(z) = 0$, $f(x,y) = \pm x^2\pm y^2$ near $z$, and in this coordinates $F$ has the form $F(x,y) = -f'_y \partial/\partial x + f'_x\partial/\partial y.$
\end{itemize}
Fix a Hamiltonian-like vector field $F$ of the given function $f:T^2\to S^1$, and let $\mathbf{F}_t:T^2\to T^2$, $t\in\RRR$ be its flow. The set $\mathsf{W}$ does not consist singular points of $\mathbf{F}$ and it  is $\mathbf{F}$-invariant and consists of periodic orbits. So one can assume that periods of all trajectories of $\mathbf{F}$ are equal to $1$ on $\mathsf{W}$.

Let $\theta_i:T^2\to T^2$ be a slide along $C_i$ supported on $W_i-V_i$, $i = 0,1,\ldots, n-1$, and set $\theta = \theta_0\circ \theta_1\circ \ldots\circ \theta_{n-1}.$ We proved that there exists a smooth function $\sigma:T^2\to \RRR$ which satisfies
\begin{itemize}
	\item $\sigma$ is constant along trajectories of $\mathbf{F},$
	\item $\sigma = 1$ on $\mathsf{V}$, $\sigma = 0$ on $T^2-\mathsf{W}$, and
	\item $\theta = \mathbf{F}_{\sigma}$,
\end{itemize}
and therefore $\theta^k = \mathbf{F}_{k\sigma}$, see \cite[ Lemma 5.2]{MaksymenkoFeshchenko:2015:HomPropCycleNonTri}. A free abelian group generated by $\theta$ will be denoted by $\langle\theta\rangle$.

\subsection{Characterization of direct and wreath products}\label{subsec:algebraic-preliminaries}
In this paragraph we recall conditions when the group $G$ splits into a direct product of its subgroups and discuss when $G$ splits into a wreath product as in subsection \ref{subsec:wreath}.
Let $G$ be a group and $G_1,\ldots, G_n$ be their subgroups. It is well-known that the group $G$ splits into a direct product $G_1\times G_2\times\ldots \times G_n$ if the following three conditions satisfied:
\begin{enumerate}
	\item[(D1)] $G_i\cap G_j = \{e\}$, for $i\neq j = 1,2,\ldots, n$, where $e$ is the unit of $G$,
	\item[(D2)] $G_iG_j = G_jG_i$ for all $i,j = 1,2,\ldots, n,$
	\item[(D3)] groups $G_1$, $G_2$ and $G_n$ generate $G$.
\end{enumerate}

The following lemma  gives  conditions when the group $G$ splits into a wreath product $L_0\wr_n\ZZZ$ for some subgroup $L_0\subset G.$
\begin{lemma}[Lemma 2.3 \cite{Maksymenko:DefFuncI:2014}]\label{lm:charact:GwrZZZ}
	Let $\phi:G\to \ZZZ$ be an epimorphism and $L_0$ be a subgroup of $\ker\phi.$ Let also 	$g\in G$ be with $\phi(g) = 1$. Assume that for  some $n\in\ZZZ$ the following conditions hold:
	\begin{enumerate}
		\item $g^n$ commutes with $\ker \phi$,
		\item $\ker\phi$ splits into a direct product of 
		$$
		L_0, \quad L_1 = g^{-1}L_0g^{1}, \quad\ldots, \quad L_{m-1} = g^{-(n-1)}L_0g^{n-1}.
		$$
	\end{enumerate}
	Then the map $\xi:L_0\wr_n\ZZZ\to G$ given by the formula
	\begin{align}\label{eq:xi-isomorphism}
		\xi(b_0,b_1,\ldots, b_{n-1}, k) &= b_0(g^{-1}b_1g^1) (g^{-2}b_2g^{2})\ldots (g^{-n+1}b_{n-1}g^{n-1})g^k \nonumber  \\
		&=b_0g^{-1}b_1\ldots g^{-1}b_{n-1}b^{-1+n+k}
	\end{align}
	is an isomorphism.
\end{lemma}

\section{Proof of Theorem \ref{thm:main-th}}\label{sec:Proof-main}
\subsection{Structure of the proof} Our main  proof-tool  is Lemma \ref{lm:charact:GwrZZZ}. So we need to define a data of  a ``natural'' epimorphism $\phi:\pi_1\mathcal{O}_f(f)\to \ZZZ$, an element $g$ from $\ker\phi$, and groups $L_i$, $i = 0,1,\ldots,n-1$ such in Lemma \ref{lm:charact:GwrZZZ}.
\subsection{Epimorphism $\phi$ and its kernel}
Let $V_i$ and $W_i$ be fixed $f$-regular neighborhoods of $C_i\in \mathcal{C}$, $i = 0,1,\ldots, n-1$ such in Subsection \ref{ss:f-adapter-neighborhood}.
\begin{proposition}[Theorem 6.1 \cite{MaksymenkoFeshchenko:2015:HomPropCycleNonTri}]\label{proposition:phi-kernel}
	There exists an epimorphism $\phi:\pi_1\mathcal{O}_f(f)\to \ZZZ$ with the kernel   isomorphic to $\pi_0\mathcal{S}'(f,\mathsf{W})$, i.e., the following  sequence of groups
	$$
	\xymatrix{
		\pi_0\mathcal{S}'(f,\mathsf{W})  \ar@{^{(}->}[r] & \pi_1\mathcal{O}_f(f) \ar@{->>}[r]^{~~\phi} & \ZZZ	
	}
	$$
	is exact.
\end{proposition}
\begin{proof}
	This result is proved in \cite{MaksymenkoFeshchenko:2015:HomPropCycleNonTri}, but for completeness of our exposition  we will recall the construction of an epimorphism $\phi$.
	Let $q:\RRR\times S^1\to S^1\times S^1 =  T^2$ be a covering map given by  $q(x,y) = (\frac{x}{n}\mod 1,y)$. Then $q(\{i\}\times S^1) = C_{i\,\mod n}$, $i\in \ZZZ$ and $q^{-1}(\mathcal{C}) = \ZZZ\times S^1.$
	
	Let $\omega:[0,1]\to \mathcal{O}_f(f)$ be a loop and $h:T^2\times [0,1]\to T^2$ be an isotopy such that $\omega_t = f\circ h_t$ and $h_0 = \id_{T^2}$, $h_1\in\mathcal{S}'(f)$. 
	There exist an isotopy $\tilde{h}:(\RRR\times S^1)\times [0,1] = \RRR\times S^1$ such that $\tilde{h}_0 = \id_{\RRR\times S^1}$ and $q\circ \tilde{h}_t = h_t\circ q$ for all $t\in [0,1].$ Since $h_1(\mathcal{C}) = \mathcal{C}$ then from the definition of $q$ we have $\tilde{h}_1(\ZZZ\times S^1) = \ZZZ\times S^1$. Then there exists an integer $\phi_{h}$ such that 
	\begin{equation}\label{eq:phi}
		\tilde{h}_1(\{i\}\times S^1) = \{i + \phi_{h} \}\times S^1.
	\end{equation}
	The number $\phi_{h}$ depends only on the homotopy class of $h$,  so on the isotopy class of the loop $\omega$, and the correspondence 
	$$
	\xymatrix{
		[\omega]\ar[rr]^{\phi} && \phi_{h}
	}
	$$
	defined by \eqref{eq:phi}  is an epimorphism $\phi:\pi_1\mathcal{O}_f(f)\to \ZZZ$.
	
	The kernel of $\phi$ consists of homotopy classes of isotopies $h:T^2\times [0,1]\to T^2$ such that $h_1$ leaves invariant each curve $C_i$ from $\mathcal{C}$, i.e., $h_1(C_i) = C_i.$ It was shown \cite[Lemma 4.14]{Maksymenko:AGAG:2006}
	that each such $h_1$ can be isotoped in $\mathcal{S}'(f)$ to the diffeomorphism $h_1'$ which is fixed on an $f$-cylindrical neighborhood $\mathsf{W}$ of $\mathcal{C}$, so $h'_1\in \mathcal{S}'(f,\mathsf{W}).$ 
	By (4) of Theorem \ref{thm:homotopy-gneral} and equation \eqref{eq:varkappa} the composition 
	$$
	\xymatrix{
		\pi_1(\mathcal{D}_{\id}(T^2, \mathsf{W}), \mathcal{S}'(f,\mathsf{W})) \ar[r]^{~~~~~~~~~~~\varkappa^{-1}}_{~~~~~~~~~\cong} & \pi_1\mathcal{O}_f(f,\mathsf{W}) \ar[r]^{\partial}_{\cong} & \pi_0\mathcal{S}'(f,\mathsf{W})
	}
	$$
	is an isomorphism. So the kernel of $\phi$ is isomorphic to $\pi_0\mathcal{S}'(f,\mathsf{W}).$
\end{proof}

\subsection{Special isotopy and subgroups of $\ker\phi$} The following proposition holds true.
\begin{proposition}[Theorem 6.1 (c), \cite{MaksymenkoFeshchenko:2015:HomPropCycleNonTri}]\label{prop:phi-g-isotopy}
	There exist an isotopy $g:T^2\times [0,1]\to T^2$ satisfying the following conditions
	\begin{enumerate}
		\item $g_1\in\mathcal{S}'(f,\mathsf{W})$,
		\item $g_1^n = \id_{T^2}$,
		\item $g_1(Q_i) = Q_{i+1}$, $i = 0,1,\ldots, n-1$, 
		\item $\phi([f\circ g_t]) = 1$.
	\end{enumerate}
\end{proposition}
\begin{proof}
	By definition of the set $\mathcal{C}$, there exists $g_1\in \mathcal{S}'(f)$ such that $g(Q_i) = Q_{i+1}.$ So (3) is obvious. To prove (1) and (2) we need to replace $g_1$ to $\sigma = \mathbf{L}_{1/n}$ on $\mathsf{W}$ in $\mathcal{S}'(f)$, where $\mathbf{L}$ is defined in \eqref{eq:IsotopiesML}; this was done in (c) of   \cite[Theorem 6.1]{MaksymenkoFeshchenko:2015:HomPropCycleNonTri}. The resulting diffeomorphism we also denote by $g_1$. 
	
	(4) Let $g:T^2\times [0,1]\to T^2$ be an isotopy between $g_1$ and $\id_{T^2}$. So $\phi([g]) = an+1$ for some $n\in \ZZZ.$ If $a\neq 0$ then we replace $[g_t]$ by $[g_t\circ \mathbf{L}_t^{-a}]$ in order to have $\phi([f\circ g_t]) = 1.$
\end{proof}

Denote by $X_i^-$ and $X_i^+$ the following intersections $Q_i\cap W_i$, $Q_i\cap  W_{i+1}$, and set $X_i = X_i^- \cup X_i^+$, $U_i = Q_i - X_i$ for $i = 0,1,\ldots, n-1$. 
The set $X_i$ is an $f$-adapted neighborhood of the boundary $\partial Q_i$ of the cylinder $Q_i\subset T^2$, see Figure \ref{fig:neighborhoods}.

\begin{figure}[h]
	\includegraphics[width=7cm]{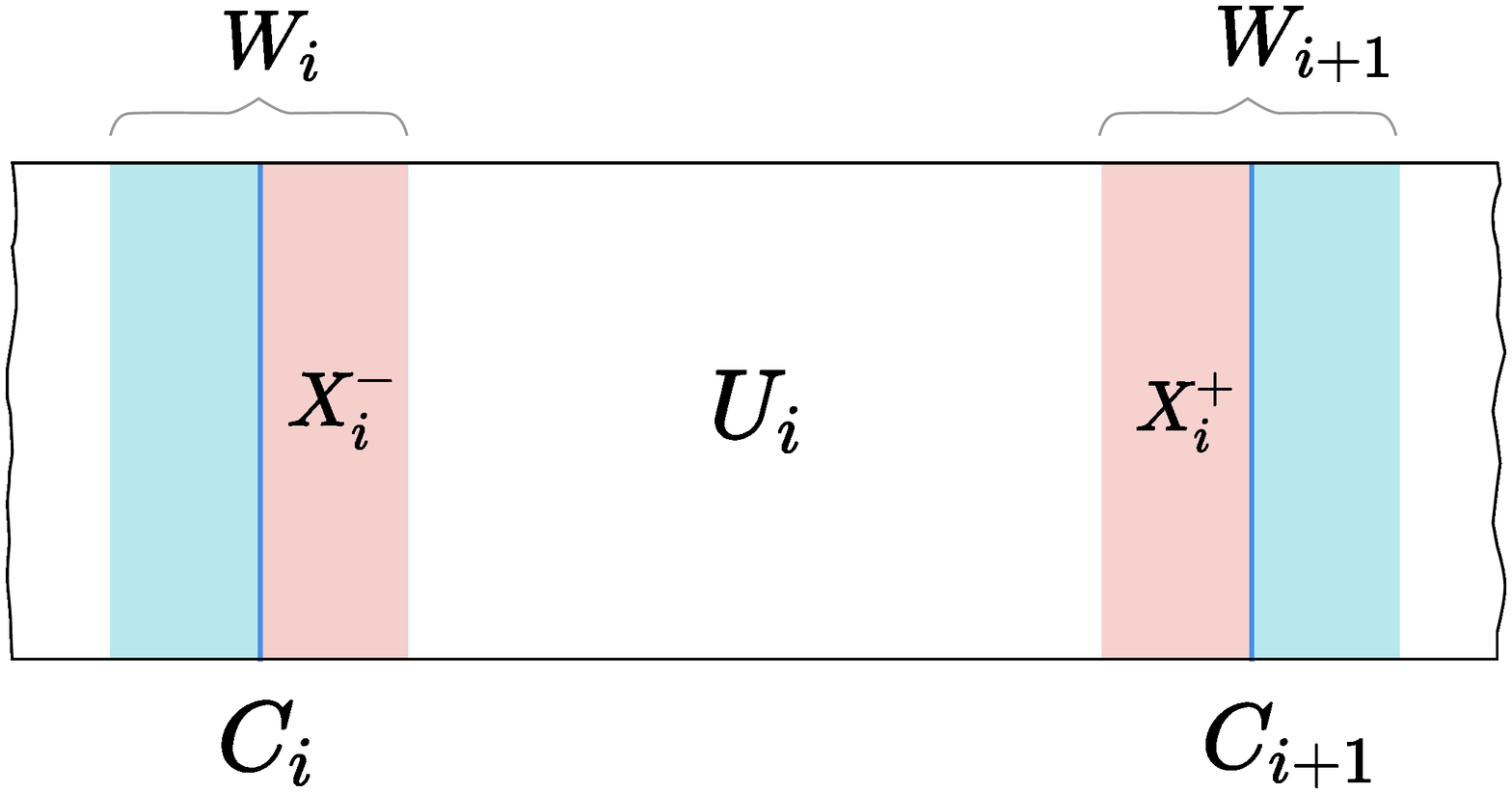}
	\caption{}
	\label{fig:neighborhoods}
\end{figure}

We denote by $L'_i$ the following subgroup of the kernel of $\phi$:
\begin{equation}\label{eq:lis}
	L'_i = \pi_0\mathcal{S}'(f, T^2 - U_i), \qquad i = 0,1,\ldots, n-1.
\end{equation}
Elements of $L'_i$ are isotopy classes of diffeomorphisms supported on $U_i$.
Let $g:T^2\times[0,1]\to T^2$ be an isotopy from Proposition \ref{prop:phi-g-isotopy}. Since $g_1$ satisfies (1)-(4) of Proposition \ref{prop:phi-g-isotopy}, it follows that 
$$
L'_i = [g_1^{-i}] L'_0 [g_1^{i}], \qquad i = 0,1,\ldots, n-1.
$$

Each diffeomorphism $h\in \mathcal{S}'(f,T^2-U_i)$ is fixed on $T^2-U_i$, so the restriction $h\to h|_{Q_i}$ induces an isomorphism
\begin{equation}\label{eq:rest}
	\beta_i:\pi_0\mathcal{S}'(f,T^2-U_i)\to \pi_0\mathcal{S}'(f|_{Q_i}, X_i)
\end{equation}
given by the  map $\beta_i([h]) = [h|_{Q_i}]$.  Let $h_i$ be a diffeomorphism from $\mathcal{S}'(f|_{Q_i}, X_i)$, and $h'_i$ be its extension to $T^2$ by the identity map. Then the inverse of $\beta_i$ is given by the rule: $\beta_i^{-1}([h_i]) = [h'_i].$
So we will not distinguish  groups $\pi_0\mathcal{S}'(f,T^2-U_i)$ and $\pi_0\mathcal{S}'(f|_{Q_i}, X_i)$,  and we set $L_i = \pi_0\mathcal{S}'(f|_{Q_i}, X_i).$

\subsection{The end of the proof}
The following lemma completes the proof.
\begin{lemma}\label{lemma:phi-lemma-main}
	The data of an epimorphism $\phi$ and the element $g$ from Proposition \ref{prop:phi-g-isotopy},  and groups $L_i =\pi_0\mathcal{S}'(f|_{Q_i}, X_i)$ from \eqref{eq:lis} satisfy conditions of Lemma \ref{lm:charact:GwrZZZ}. So  $\pi_1\mathcal{O}_f(f)$ is isomorphic to $\pi_0\mathcal{S}'(f|_{Q_0}, X_0)\wr_n\ZZZ$ and this isomorphism is given by the formula \eqref{eq:xi-isomorphism}.
\end{lemma}
\begin{proof}
	Let $g:T^2\times [0,1]\to T^2$ is an isotopy defined from Proposition \ref{prop:phi-g-isotopy}. By (2) of Proposition \ref{prop:phi-g-isotopy} the diffeomorphism $g^n_1 = \id_{T^2}$, and so it commutes with $\ker \phi.$ Thus (1) of Lemma \ref{lm:charact:GwrZZZ} holds true. 
	
	To verify (2) of  Lemma \ref{lm:charact:GwrZZZ} we need to check that the  three conditions (D1)--(D3) from subsection \ref{subsec:algebraic-preliminaries}  satisfied:
	\begin{enumerate}
		\item $L_i\cap L_j = [\id_{T^2}]$, $i\neq j$
		\item $L_iL_j = L_jL_i$,
		\item groups $L_0$, $L_1,\ldots, L_{n-1}$ generate $\ker\phi$,
	\end{enumerate}
	for all $i,j = 0,1,\ldots, n-1$. Conditions (1) and (2) follows from the fact that $\supp(h_i)\cap \supp(h_j) = \varnothing$ for $[h_i]\in  L_i$, $[h_j]\in L_j$, $i\neq j = 0,1,\ldots, n-1$.

	Let $[h]$ belongs to $\ker\phi.$ Then $h|_{Q_i}\in \mathcal{S}'(f|_{Q_i},X_i)$ and there is a unique  decomposition
	$$
	[h] = [h|_{Q_0}][h|_{Q_1}]\ldots [h|_{Q_{n-1}}].
	$$
	From the other hand let $h_i$ be a diffeomorphism from $\mathcal{S}'(f|_{Q_i}, X_i),$ $i = 0,1,\ldots, n-1$. Then 
	$$
	\beta_0^{-1}([h_0])\circ \beta_1^{-1}([h_1])\circ \ldots \circ \beta_{n-1}^{-1}([h_{n-1}])
	$$
	represents an element of $\ker\phi.$ So (3) is true i.e.,  groups $L_0,\,L_1,\ldots,L_{n-1}$ generate $\ker\phi.$
\end{proof}

\section{The kernel of $\phi$ and an isotopy $\mathbf{M}$}\label{sec:isitopyM}
Note that the group $\langle\mathbf{M}\rangle$ contains in the kernel of $\phi$, but to prove Theorem \ref{thm:main-th} we do not need the explicit form of the element of $\ker\phi$ representing $\mathbf{M}.$ 
Let $\mathsf{V}$ and $\mathsf{W}$ be fixed in subsection \ref{ss:f-adapter-neighborhood} $f$-regular neighborhoods of $\mathcal{C}.$
Next proposition describes the relationship between $\ker\phi$ and the group $\langle\mathbf{M}\rangle$. 
\begin{proposition}\label{theorem:kernel-phi-M}
	There is an isomorphism $\pi_0\mathcal{S}'(f,\mathsf{W})\cong \langle \mathbf{M} \rangle\times \pi_0\mathcal{S}'(f,\mathsf{V})$.
\end{proposition}
\begin{proof}
	In \cite{MaksymenkoFeshchenko:2015:HomPropCycleNonTri} we showed that the isotopy class of  $\mathbf{M}$ can be represented as the isotopy class of $\theta = \theta_0\circ \theta_1\circ \ldots\circ \theta_{n-1}$, where $\theta_i$ is a slide along $C_i$ supported in $W_i-V_i$, $i = 0,1,2\ldots, n-1.$ 
	
	To prove this proposition we need to check that conditions (D1)--(D3) from subsection \ref{subsec:algebraic-preliminaries} hold for subgroups $\langle\theta\rangle$ and $\pi_0\mathcal{S}'(f,\mathsf{V})$ of $\pi_0\mathcal{S}'(f,\mathsf{W}).$ Conditions (D1) and (D2) obviously hold since $\supp(\theta)\cap \supp(h) = \varnothing$, where $h\in \mathcal{S}'(f,\mathsf{V}).$
	So, it remains  to show that for each $h$ in $\mathcal{S}'(f,\mathsf{W})$ there is a unique decomposition 
	\begin{equation}\label{eq:h-decomposition}
		[h] = [\theta^{k(h)}]\circ [h'],
	\end{equation}
	where $h'\in\mathcal{S}'(f,\mathsf{V})$ and $k(h)\in\ZZZ.$ To define this decomposition we use the same arguments such in \cite[Theorem 5.5]{Feshchenko:DefFuncI:2019}.
	
	Let $h$ be a diffeomorphism from  $\mathcal{S}'(f,\mathsf{W}).$ Since $f$ is fixed on $\mathsf{W}$, it follows from  \cite[Lemma 6.1]{Maksymenko:DefFuncI:2014}, there exists a smooth function $\alpha:\mathsf{W}\to \RRR$ such that $h = \mathbf{F}_{\alpha}.$ The restriction of $h$ and $\alpha$ onto $W_i$ are denoted by $h_i$ and $\alpha_i$ respectively. Since periods of all trajectories of $\mathbf{F}|_{\mathsf{W}}$ is equal to $1$, it follows that $\alpha_i$ takes an integer number $k_i(h)\in\ZZZ$, $x\in W_i.$
	The diffeomorphism $h|_{Q_i}$ is isotopic relative $\mathsf{W}\cap Q_i$ to a Dehn twist $\tau^{a_i}$ supported on $\mathsf{W}\cap Q_i$, where $a_i = \alpha(C_{i+1})-\alpha(C_i)$, $i = 0,1,\ldots, n-1$. Since $h\in\mathcal{S}'(f,\mathsf{W})$, it follows that $h|_{Q_i}$ is isotopic to the identity map of $Q_i.$ Hence $a_i = \alpha(C_{i+1})-\alpha(C_i) = k_{i+1}(h) - k_i(h) = 0.$ Then numbers  $k_{i}(h)$ pairwise  equal for $i = 0,1,\ldots, n-1$, so they all coincide, i.e. $k_i(h) = k(h).$ 
	
	Define an isotopy $H^t:T^2\to T^2$ between $h$ and $\theta^{-k(h)}\circ h$ by the formula
	$$
	H^t(h) = \mathbf{F}^{-1}_{tk(h)\sigma}\circ h.
	$$
	A diffeomorphism $H^t(h)$ is fixed on $\mathsf{V}$ for all $t\in [0,1].$ Then we have the following decomposition
	$$
	[h] = [\theta^{k(h)}]\circ [\theta^{-k(h)}\circ h] =[\theta^{k(h)}] \circ  [H^1(h)]
	$$
	which coincides with \eqref{eq:h-decomposition} for $h' = H^1(h).$
\end{proof}

\bibliographystyle{plain}

\end{document}